\documentclass{birkmult}

\numberwithin{equation}{section}
\newtheorem{theorem}{Theorem}[section]
\newtheorem{lemma}[theorem]{Lemma}
\newtheorem{definition}[theorem]{Definition}
\newtheorem{corollary}[theorem]{Corollary}
\newtheorem{proposition}[theorem]{Proposition}
\newtheorem{remark}[theorem]{Remark}
\newtheorem{example}[theorem]{Example}

\newcommand{\Ent}{\mbox{\rm Ent}}

\newcommand{\vol}{\mbox{\rm vol}}

\newtheorem{conjecture}[theorem]{Conjecture}
\newcommand{\R}{\mathbb{R}}
\newcommand{\K}{\mathcal{K}}
\newcommand{\G}{\mathcal{G}}
\newcommand{\Conj}{{\frak C}}
\newcommand{\cc}{^{\frak c}}

\renewcommand{\labelenumi}{(\roman{enumi})}
\begin{document}

\title{Entropic Measure on Multidimensional Spaces}

\author{Karl-Theodor Sturm}

\address{%
Institut fuer Angewandte Mathematik\\
Poppelsdorfer Allee 82\\
D 53115 Bonn\\
Germany}

\email{sturm@uni-bonn.de}

\begin{abstract}
We construct the entropic measure $\mathbb{P}^\beta$ on compact manifolds of any dimension. It is defined as the push forward of the Dirichlet process (another random
probability measure, well-known to exist on spaces of any dimension) under the {\em conjugation map} $$\Conj:\mathcal{P}(M)\to\mathcal{P}(M).$$ This conjugation map
is a continuous involution. It can be regarded as the canonical extension to higher dimensional spaces of a map between probability measures on 1-dimensional spaces
characterized by the fact that the distribution functions of $\mu$ and $\Conj(\mu)$ are inverse to each other.

We also present an heuristic interpretation of the entropic measure as
$$d\mathbb{P}^\beta(\mu)=\frac{1}{Z}\exp\left(-\beta\cdot \mbox{Ent} (\mu|m)\right)\cdot
d\mathbb{P}^0(\mu).$$
\end{abstract}




\date{January 1, 2009}
\dedicatory{}

\maketitle

\section{Introduction}

Gradient flows of entropy-like functionals on the Wasserstein space turned out to be a powerful tool in the study of various dissipative PDEs on Euclidean or
Riemannian spaces $M$, the prominent example being the heat equation. See e.g. the monographs \cite{Villani03, AGS05} for more examples and further references.

In \cite{SR08}, von Renesse and the author  presented an approach to stochastic perturbation of the gradient flow of the entropy. It is based on the construction of
a
Dirichlet form
 \begin{eqnarray*}
\mathcal{E}(u,u)&=& \int_{\mathcal{P}(M)} \|\nabla u\|^2(\mu) \ d\mathbb{P}^\beta(\mu)\end{eqnarray*}
where $\|\nabla u\|$ denotes the norm of the gradient in the Wasserstein space $\mathcal{P}(M)$ as introduced by Otto \cite{Otto01}. The fundamental new ingredient
was the measure $\mathbb{P}^\beta$ on the Wasserstein space. This so-called {\em entropic measure} is an interesting and challenging object in its own right.
It is formally introduced as
\begin{equation}d\mathbb{P}^\beta(\mu)=\frac1Z\exp\left(-\beta\cdot \mbox{Ent} (\mu|m)\right)\cdot
d\mathbb{P}^0(\mu)
\end{equation}
with some (non-existing) `uniform distribution' $\mathbb{P}^0$ on the Wasserstein space $\mathcal{P}(M)$ and the relative entropy
 as a potential.

A rigorous construction was presented for 1-dimensional spaces. In the case $M=[0,1]$ it is based on the bijections
\begin{equation*}
\mu\quad \stackrel{\underleftrightarrow{\scriptstyle{(x)=\mu([0,x])}}}{} \quad f
\quad \stackrel{\underleftrightarrow{\quad \scriptstyle{g=f^{(-1)}}\quad}}{} \quad g
\quad \stackrel{\underleftrightarrow{\scriptstyle {g(y)=\nu([0,y])}}}{} \quad \nu
\end{equation*}

between
{\em probability measures, distribution functions} and {\em inverse distribution functions}
(where $f^{(-1)}(y)=\inf\{x\ge 0: f(x)\ge y\}$ more precisely denotes the `right inverse' of $f$).
If $\Conj: \mathcal{P}(M)\to\mathcal{P}(M)$ denotes the map $\mu\mapsto\nu$ then the entropic measure $\mathbb{P}^\beta$ is just the push forward under $\Conj$ of
the Dirichlet-Ferguson process $\mathbb{Q}^\beta$. The latter is a random probability measure which is well-defined on every probability space.

\medskip

For long time it seemed that the previous construction is definitively limited to dimension 1 since it heavily depends on the use of distribution functions (and
inverse distribution functions), -- objects which do not exist in higher dimensions.
The crucial observation to overcome this restriction is to interpret $g$ as the unique {\em optimal transport map} which pushes forward $m$ (the normalized uniform
distribution on $M$) to $\mu$:
$$\mu=g_*m.$$
Due to Brenier \cite{Brenier87} and McCann \cite{McCann01} such a `monotone map' exists for each probability measure $\mu$ on a Riemannian manifold of arbitrary
dimension.
Moreover, also in higher dimensions such a monotone map $g$ has a unique generalized inverse $f$, again being a monotone map (with generalized inverse being $g$).
This observation allows to define the {\em conjugation map}
$$\Conj: \mathcal{P}(M)\to\mathcal{P}(M), \ \mu\mapsto\nu$$ for any compact manifold $M$. It is a continuous involution.
By means of this map we define the entropic measure as follows:
$$\mathbb{P}^\beta:= \Conj_*\mathbb{Q}^\beta$$
where $\mathbb{Q}^\beta$ denotes the Dirichlet-Ferguson process on $M$ with intensity measure $\beta\cdot m$. (Actually, such a random probability
measure exists on every probability space.)

\medskip

In order to justify our definition of the entropic measure by some heuristic argument
let us assume that  $\mathbb{P}^\beta$ were given as in (1.1). The identity
$\mathbb{Q}^\beta= \Conj_*\mathbb{P}^\beta$
then defines a probability measure which satisfies
\begin{equation}
d\mathbb{Q}^\beta(\nu)=\frac1Z\exp\left(-\beta\cdot \mbox{Ent} (m|\nu)\right)\cdot
d\mathbb{Q}^0(\nu).\end{equation}

Given a measurable partition $M = {\bigcup}_{i=1}^N M_i$
and approximating arbitrary probability measures $\nu$ by
measures with constant density on each of the sets $M_i$ of the partition
the previous ansatz (1.2) yields -- after some manipulations --
\begin{align*}
  &\mathbb{Q}^\beta_{M_1,\ldots,M_N} (dx)  \\
&= \frac{\Gamma(\beta)}{\overset{N}{\underset{i=1}{\prod}} \Gamma (\beta m (M_i))} \cdot x_1^{\beta \cdot m (M_1)-1} \cdot \ldots \cdot x_{N-1}^{\beta \cdot m
(M_{N-1})-1} \cdot x_N^{\beta \cdot m (M_N)-1} \times \\
 &\quad\times \delta_{(1-{\overset{N-1}{\underset{i=1}{\sum}}} x_i)}(dx_N) dx_{N-1} \ldots dx_1.
\end{align*}
These are, indeed, the finite dimensional distributions of
 the Dirichlet-Ferguson process.

\section{Spaces of Convex Functions and Monotone Maps}
Throughout this paper, $M$ will be a compact subset of a complete Riemannian manifold $\hat M$ with Riemannian distance $d$ and $m$ will denote a probability measure
with support  $M$, absolutely continuous  with respect to the volume measure. We assume that it satisfies a Poincar\'e inequality: $\exists c>0$
$$\int_M |\nabla u|^2\,dm\ge c\cdot\int_Mu^2\,dm$$
for all weakly differentiable $u:M\to \mathbb{R}$ with $\int_M u\,dm=0$.

For compact Riemannian manifolds, there is a canonical choice for $m$, namely, the normalized Riemannian volume measure. The freedom to choose $m$ arbitrarily might
be of advantage in view of future extensions: For Finsler manifolds and for non-compact Riemannian manifolds there is no such canonical probability measure.

\medskip

The main ingredient of our construction below will be the Brenier-McCann representation of optimal transport in terms of gradients of convex functions.
\begin{definition}
A function $\varphi: M \rightarrow \R $ is called $d^2/2$-convex if there exists a function $\psi: M \rightarrow \R $ such that
\begin{displaymath}
 \varphi(x)=-\underset{y\in M}{\inf} \left[\frac{1}{2} d^2(x,y)+\psi(y)\right]
\end{displaymath}
for all $x \in M$. In this case, $\varphi$ is called \emph{generalized Legendre transform} of $\psi$ or \emph{conjugate} of $\psi$ and denoted by
$$\varphi=\psi\cc.$$
\end{definition}

Let us summarize some of the basic facts on $d^2/2$-convex functions. See \cite{Rockafellar70}, \cite{Rueschendorf96}, \cite{McCann01} and \cite{Villani08} for
details.\footnote{A function $\varphi$ is $d^2/2$-convex in our sense if and only if the function $-\varphi$ is $c$-concave in the sense of
\cite{Rockafellar70, Rueschendorf96, McCann01,Villani08} with cost function $c(x,y)=d^2(x,y)/2$.
In our presentation, the ${}\cc$ stands for `conjugate'. For the relation between $d^2/2$-convexity and usual convexity on Euclidean space we refer to chapter 4.}

\begin{lemma}
\begin{enumerate}
 \item A function $\varphi$ is $d^2/2$-convex if and only if $${\varphi\cc}\cc = \varphi.$$
 \item Every $d^2/2$-convex function is bounded, Lipschitz continuous and differentiable almost everywhere with gradient bounded by $D= \underset{x,y \in
     M}{\sup}
     d(x,y)$.
\end{enumerate}
\end{lemma}

In the sequel, ${\K}={\K}(M)$ will denote the set of $d^2/2$-convex functions on $M$ and $\tilde{\K} = \tilde{\K}(M)$ will denote the set of equivalence classes
in  ${\K}$ with $\varphi_1 \sim \varphi_2$ iff $\varphi_1 - \varphi_2$ is constant.
${\K}$ will be regarded as a subset of the Sobolev space $H^1(M,m)$ with norm
$$\parallel u \parallel_{H^1} = \left[ \int_M \mid \bigtriangledown u\mid^2 dm + \int_M u^2 dm \right]^\frac{1}{2}$$
and $\tilde{\K} = {\K}/const $ will be regarded as a subset of the space $\tilde{H^1}= H^1 / const$ with norm
$$\parallel u \parallel_{\tilde{H^1}} = \left[ \int_M \mid \nabla u\mid^2 dm \right]^ \frac{1}{2}.$$

\begin{proposition}
For each Borel map $g:M\rightarrow M$ the following are equivalent:
\begin{enumerate}
 \item $\exists \varphi \in \tilde\K : g= \exp(\nabla \varphi)$ a.e. on $M$;
 \item $g$ is an optimal transport map from $m$ to $f_\ast m$ in the sense that it is a minimizer of $h \mapsto \int_M d^2(x,h(x))m(dx)$ among all Borel maps
     $h:M\rightarrow M$ with $h_\ast m = g_\ast m$.
\end{enumerate}

\end{proposition}
In this case, the function $\varphi \in \tilde \K$ in (i) is defined uniquely. Moreover, in (ii) the map $f$ is the unique minimizer of the given minimization
problem.

A Borel map $g:M\rightarrow M$ satisfying the properties of the previous proposition will be called \emph{monotone map} or \emph{optimal Lebesque transport}.
The set of $m$-equivalence classes of such maps will be denoted by $\G=\G(M)$.
Note that $\G(M)$ does \emph{not depend} on the choice of $m$ (as long as $m$ is absolutely continuous with full support)!
$\G(M)$ will be regarded as a subset of the space of maps $L^2((M,m)(M,d))$ with metric $d_2(f,g) = \left[ \int_M d^2(f(x),g(x))m(dx) \right]^\frac{1}{2} $.

According to our definitions, the map $\Upsilon:\varphi \mapsto \exp(\nabla \varphi)$ defines a bijection between $\tilde \K$ and $\G$.
Recall that $\mathcal{P} = \mathcal{P}(M)$ denotes the set of probability measures $\mu$ on $M$ (equipped with its Borel $\sigma$-field).

\begin{proposition}
 The map $\chi: g \mapsto g_\ast m$ defines a bijection between $\G$ and $\mathcal{P}(M)$. That is, for each $\mu \in \mathcal{P}$ there exists a unique $g \in \G$
 --
 called \emph{Brenier map} of $\mu$ -- with $\mu = g_\ast m$.
\end{proposition}

The map $\chi$ of course strongly depends on the choice of the measure $m$. (If there is any ambiguity we denote it by $\chi_m$.)

Due to the previous observations, there exist canonical bijections $\Upsilon$ and $\chi$ between the sets $\tilde \K$, $\G$ and $\mathcal{P}$.
Actually, these bijections are even homeomorphisms with respect to the natural topologies on these spaces.

\begin{proposition}
Consider any sequence $\left\lbrace \varphi_n \right\rbrace_{n \in \mathbb{N}} $ in $\tilde \K$ with corresponding sequences $\left\lbrace g_n \right\rbrace_{n \in
\mathbb{N}} = \left\lbrace \Upsilon(\varphi_n) \right\rbrace_{n \in \mathbb{N}}$ in $\G$ and $\left\lbrace \mu_n \right\rbrace_{n \in \mathbb{N}} = \left\lbrace
\chi(g_n) \right\rbrace_{n \in \mathbb{N}}$ in $\mathcal{P}$ and let $\varphi \in \tilde \K$, $g=\Upsilon(\varphi) \in \G$, $\mu =\chi(g) \in \mathcal{P}$. Then the
following are equivalent:
\begin{enumerate}
 \item $\varphi_n \longrightarrow \varphi$ in $\tilde H_1$
 \item $g_n \longrightarrow  g$ in $L^2((M,m),(M,d))$
 \item $g_n \longrightarrow  g$ in $m$-probability on $M$
 \item $\mu_n \longrightarrow  \mu$ in $L^2$-Wasserstein distance $d_W$
 \item $\mu_n \longrightarrow  \mu$ weakly.
\end{enumerate}

\end{proposition}

\begin{proof}
$ (i)\Leftrightarrow (ii)$
Compactness of $M$ and smoothness of the exponential map imply that there exists $\delta > 0$ such that $ \forall x \in M$, $\forall v_1, v_2 \in T_x M $ with $\mid
v_1\mid, \mid v_2 \mid \leq D$ and $\mid v_1 - v_2 \mid < \delta$:
$$ \frac{1}{2} \leq d(exp_x v_1, exp_x v_2)/ \mid v_1 - v_2 \mid_{T_x M} \leq 2.$$
Hence, $\varphi_n \longrightarrow  \varphi$ in $\tilde H^1$, that is
$\int_M \mid \nabla \varphi_n(x) - \nabla \varphi(x) \mid_{T_x M}^2 m(dx) \longrightarrow 0$,
is equivalent to
$\int_M d^2(g_n(x),g(x)) m(dx) \longrightarrow 0$,
that is, to $g_n\longrightarrow  g$ in $L^2((M,m),(M,d))$.

$ (ii)\Leftrightarrow (iii)$
Standard fact from integration theory (taking into account that $d(g_n,g)$ is uniformly bounded due to compactness of $M$).

$ (ii)\Leftrightarrow (iv)$
If $\mu_n = (g_n)_\ast m$ and $\mu_n = g_\ast m$ then $(g_n,g)_\ast m$ is a coupling of $\mu_n$ and $\mu$. Hence,
\begin{equation} {\label {dW2}}
 d_W ^2 (\mu_n,\mu) \leq \int_M d^2(g_n(x),g(x)) m(dx).
\end{equation}

$ (iv)\Leftrightarrow (v)$
Trivial.

$ (ii)\Leftrightarrow (iv)$
\cite{Villani08}, Corollary 5.21.

\end{proof}

\begin{remark}\rm
 Since $M$ is compact, assertion (ii) of the previous Proposition is equivalent to

(iii') $g_n \longrightarrow g$ in $L^p ((M,m),(M,d))$

for any $p \in [1,\infty)$ and similarly, assertion (iv) is equivalent to

(iv') $\mu_n \longrightarrow \mu$ in $L^p$-Wasserstein distance.
\end{remark}

\begin{remark}\rm
 In $n=1$, the inequality in (\ref{dW2}) is actually an equality. In other words, the map
\begin{equation*}
 \chi : (\G,d_2)\rightarrow(\mathcal{P},d_W)
\end{equation*}
is an \emph{isometry}. This is no longer true in higher dimensions.
\end{remark}

The well-known fact (Prohorov's theorem) that the space of probability measures on a compact space is itself compact, together with the previous continuity results
immediately implies compactness of $\tilde \K$ and $\G$.

\begin{corollary}
 (i) $\tilde \K$ is a compact subset of $\tilde H^1$.

(ii) $\G$ is a compact subset of $L^2((M,m),(M,d))$.
\end{corollary}

\section{The Conjugation Map}

Let us recall the definition of the conjugation map $ \Conj_\K : \varphi \mapsto \varphi \cc$
acting on functions $\varphi : M\rightarrow \R$ as follows
\begin{equation*}
 \varphi\cc (x) = -\underset{y\in M}{\inf} \left[ \frac{1}{2} d^2(x,y) + \varphi(y) \right] .
\end{equation*}
The map $\Conj_\K$ maps $\K$ bijective onto itself with
$\Conj^2 _\K = Id$. For each $\lambda \in \R$, $\Conj_\K (\varphi + \lambda) = \Conj_\K (\varphi) - \lambda$.
Hence, $\Conj _\K$ extends to a bijection $\Conj_{\tilde \K}: \tilde \K \rightarrow \tilde \K$.
Composing this map with the bijections $\chi:\G \rightarrow \mathcal{P}$ and $\Upsilon : \tilde \K \rightarrow \G$ we obtain involutive bijections
\begin{equation*}
 \Conj_\G = \Upsilon \circ \Conj_{\tilde \K} \circ \Upsilon^{-1}:\G \rightarrow \G
\end{equation*}
and
\begin{equation*}
  \Conj_\mathcal{P} = \chi \circ \Conj_{\G} \circ \chi^{-1}:\mathcal{P} \rightarrow \mathcal{P},
\end{equation*}
called \emph{conjugation map} on $\G$ or on $\mathcal{P}$, respectively.
Given a monotone map $g \in \G$, the monotone map
\begin{equation*}
 g\cc := \Conj_\G (g)
\end{equation*}
will be called \emph{conjugate map} or \emph{generalized inverse map}; given a probability measure $\mu \in \mathcal{P}$ the probability measure
\begin{equation*}
 \mu\cc := \Conj_\mathcal{P} (\mu)
\end{equation*}
will be called \emph{conjugate measure}.

\begin{example}
 (i) Let $M = S^n$ be the $n$-dimensional sphere, and $m$ be the normalized Riemannian volume measure.
Put $$\mu= \lambda \delta_a + (1-\lambda)m$$ for some point $a \in M$ and $\lambda \in \enspace ]0,1[$. Then $$\mu \cc = \frac{1}{1-\lambda} 1_{M\setminus B_r(a)}
\cdot m$$
where $r>0$ is such that $m(B_r(a)) = \lambda$.

{\footnotesize\rm [ Proof. The optimal transport map $g=\exp(\nabla\varphi)$ which pushes $m$ to $\mu$ is determined by the $d^2/2$-convex function
$$\varphi=\left\{\begin{array}{ll}
\frac12\left[r^2-d^2(a,x)\right]\, &\quad\mathrm{in } \  B_r(a)\\
\frac{r}{2(\pi-r)}\left[d^2(a',x)-(\pi-r)^2\right]\, &\quad\mathrm{in } \ \overline B_{\pi-r}(a')=M\setminus B_r(a)
\end{array}\right.$$
Its conjugate is
the function
$$\varphi\cc(y)=-\frac{r}{2\pi}d^2(a',y)+\frac12r(\pi-r). \quad ]$$
}

\begin{flushleft}
\includegraphics[width=0.337\textwidth]{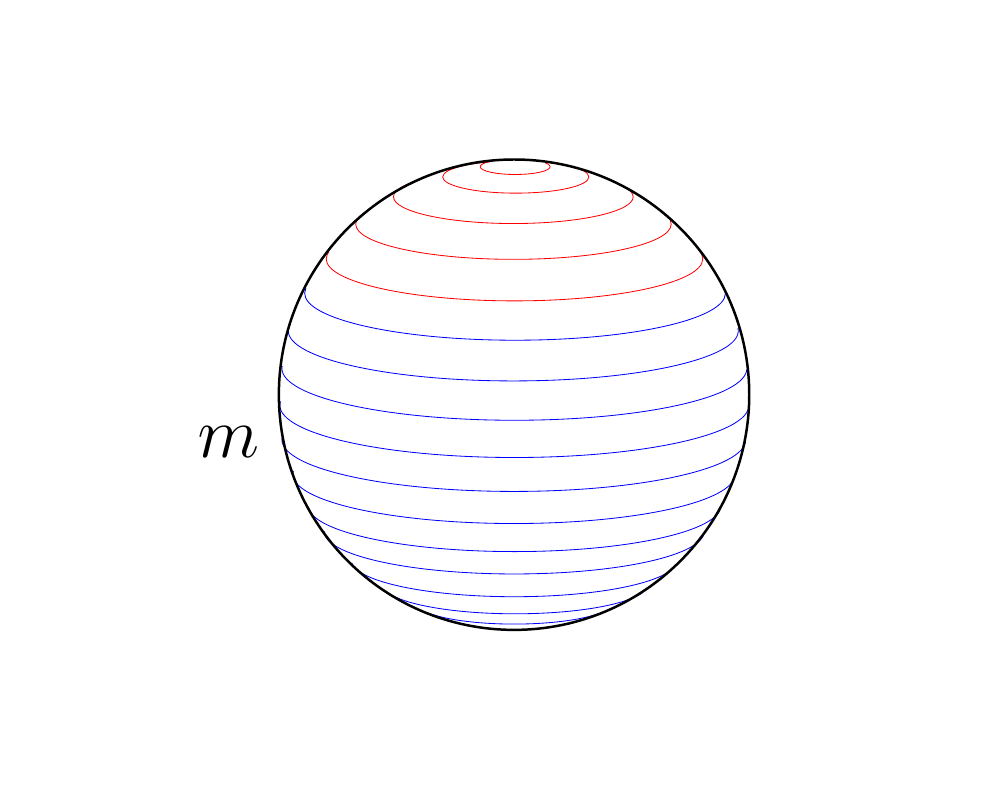} \vspace{-0.5cm} \hspace{-0.49cm}
\includegraphics[width=0.331\textwidth]{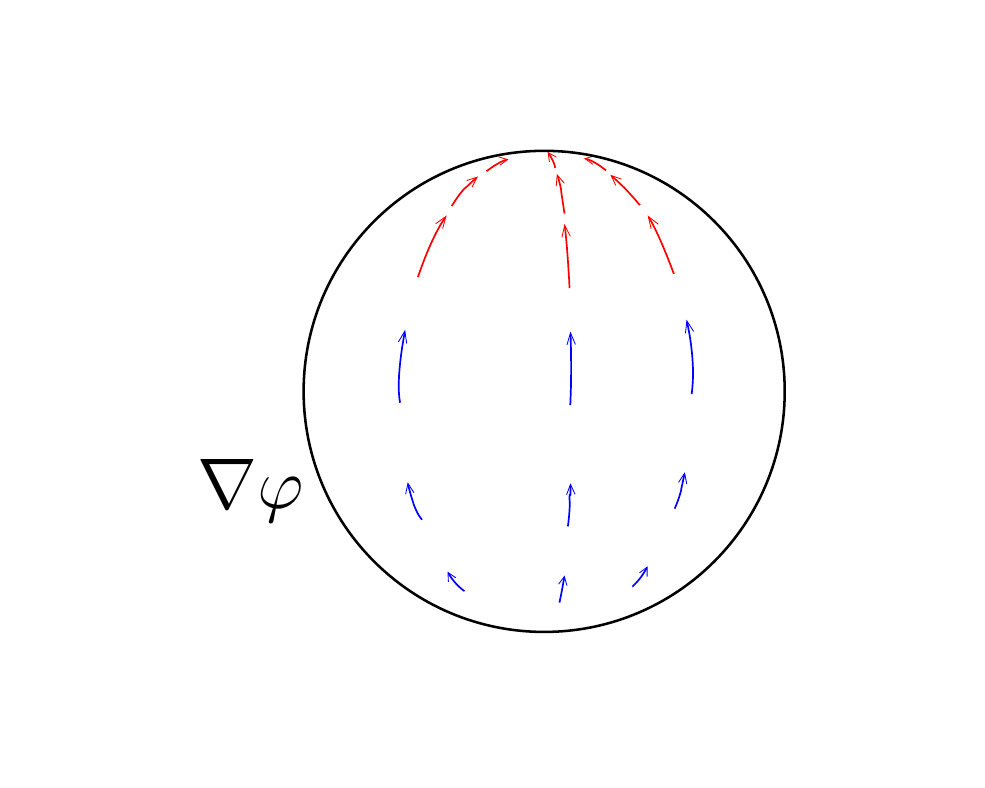}
\includegraphics[width=0.328\textwidth]{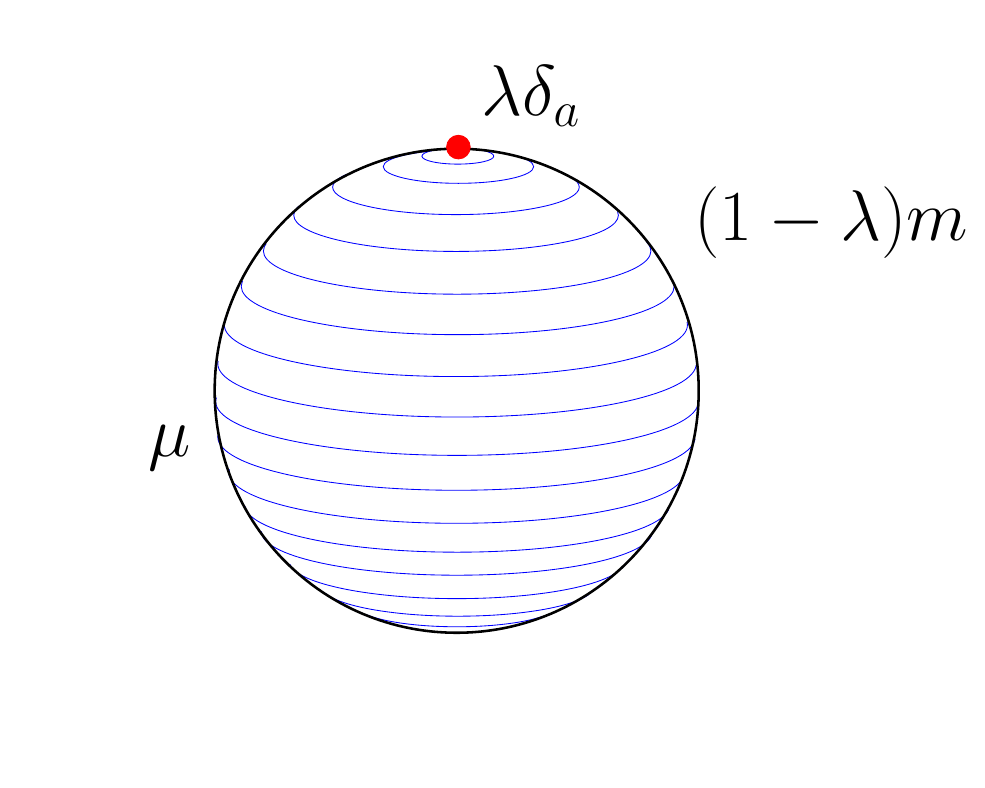}
\includegraphics[width=0.334\textwidth]{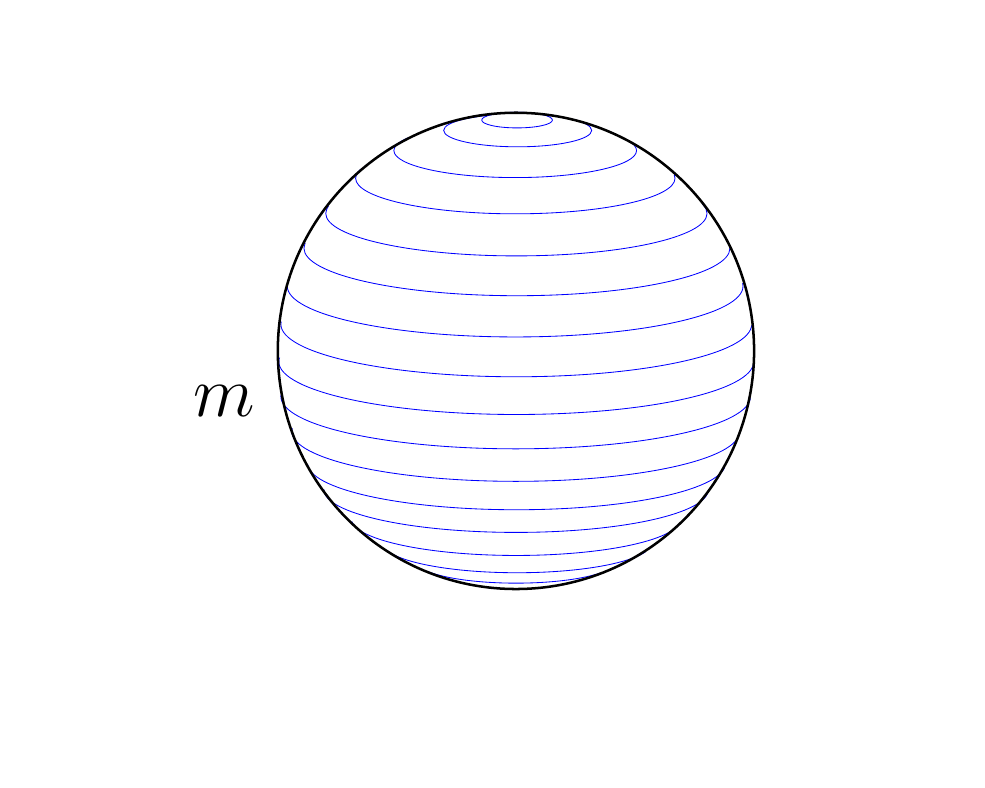} \hspace{-0.52cm}
\includegraphics[width=0.330\textwidth]{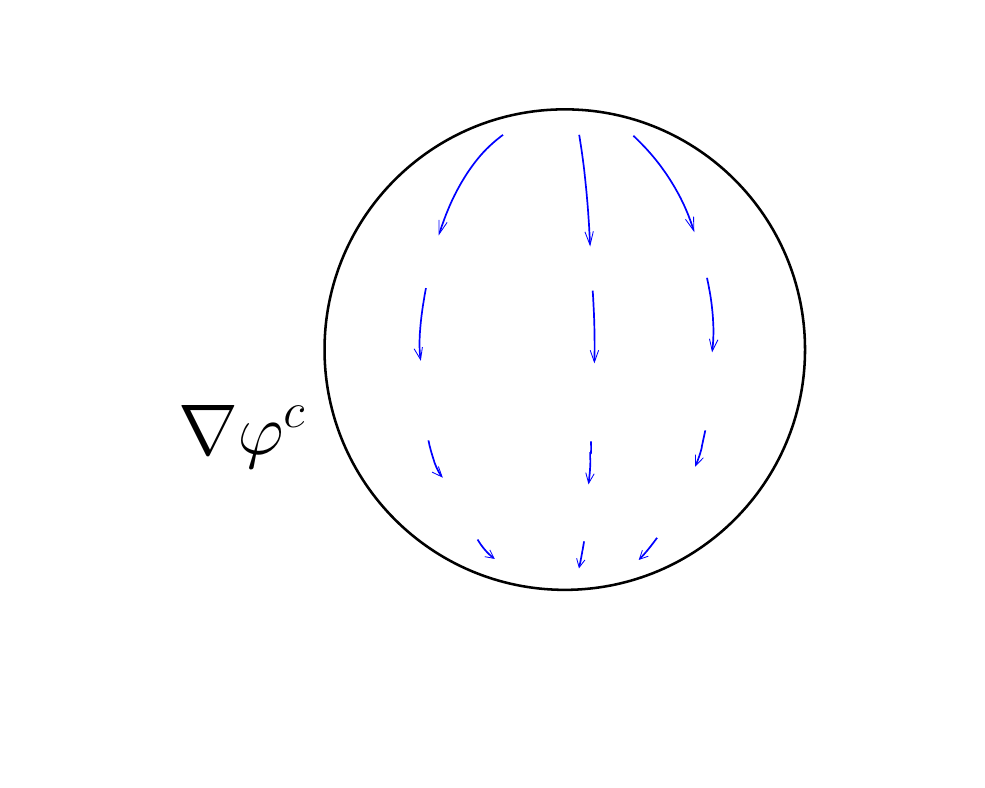} \hspace{-0.4cm}
\includegraphics[width=0.334\textwidth]{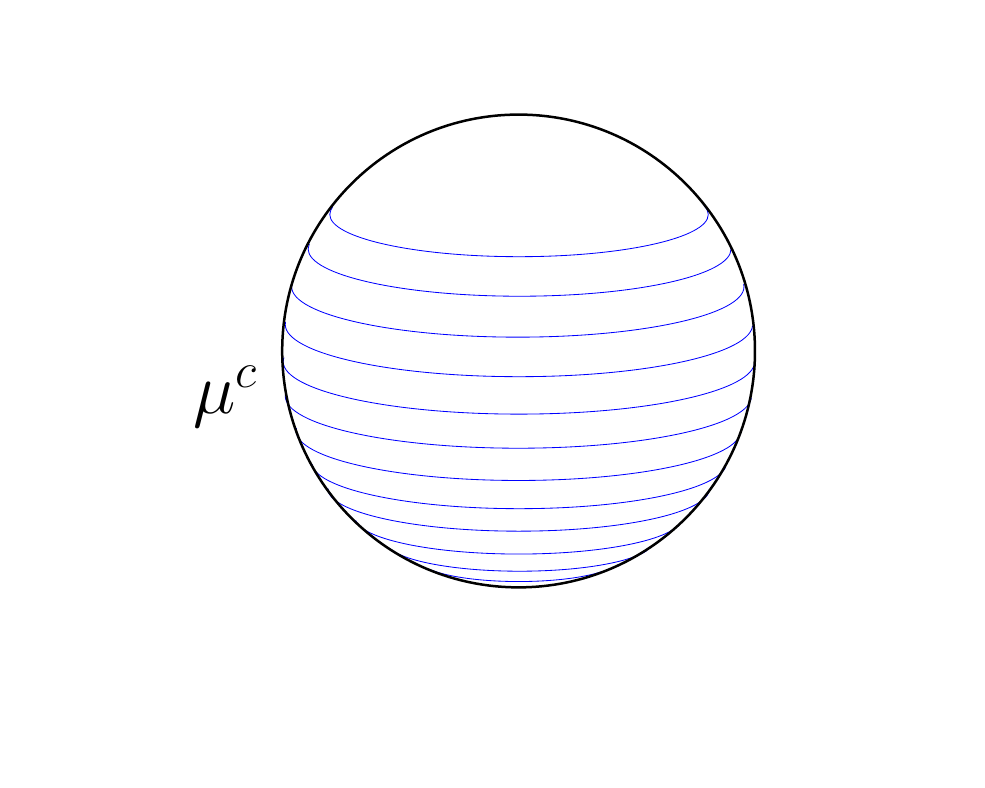}
\end{flushleft}

 (ii) Let $M= S^n$, the $n$-dimensional sphere, and $\mu=\delta_a$ for some $a\in M$. Then $\mu\cc =\delta_{a'}$ with $a'\in M$ being the antipodal point of a.

{\footnotesize\rm [ Proof. Limit of (i) as $\lambda\nearrow 1$. Alternatively: explicit calculations with
$\varphi(x)=\frac12 [\pi^2-d^2(a,x)]$ and
\begin{align*}
\varphi\cc(y)=\sup_x\left(-\frac12 d^2(x,y)+\frac12 d^2(a,x)-\frac12\pi^2\right)=-\frac12 d^2(a',y).\quad]\end{align*}}

\medskip

(iii) Let $M= S^n$, the $n$-dimensional sphere, and $\mu=\frac12\delta_a+\frac12\delta_{a'}$ with north and south pole $a,a'\in M$. Then $\mu\cc$ is the uniform
distribution on the equator, the $(n-1)$-dimensional set $Z$ of points of equal distance to $a,a'$.

\medskip
(iv) Let $M= S^1$ be the circle of length 1, $m$ = uniform distribution and
$$\mu = \sum_{i=1}^{k} \alpha_i \delta_{x_i}$$
with points $x_1< x_2 < \ldots< x_k< x_1$ in cyclic order on $S^1$ and numbers $\alpha_i \in [0,1]$, $\sum \alpha_i =1$. Then $$\mu \cc= \sum_{i=1}^{k} \beta_i
\delta_{y_i}$$
with $\beta_i = \mid x_{i+1}- x_i \mid$ and points $y_1< y_2 < \ldots< y_k< y_{k+1}= y_1$ on $S^1$ satisfying $$\mid y_{i+1}-y_i\mid = \alpha_{i+1}.$$
{\footnotesize\rm  [ Proof. Embedding in $\mathbb{R}^1$ and explicit calculation of distribution and inverse distribution functions. ]}
\end{example}

\begin{remark}\label{two-m}\rm
 The conjugation map
\begin{equation*}
 \Conj_\mathcal{P} : \mathcal{P}\rightarrow \mathcal{P}
\end{equation*}
depends on the choice of the reference measure $m$ on $M$. Actually, we can choose two different probability measures $m_1$, $m_2$ and consider
$  \Conj_\mathcal{P} = \chi_{m_2} \circ \Conj_{\G} \circ \chi^{-1}_{m_1}$.
\end{remark}

\begin{proposition}
 Let $\mu =g_\ast m \in \mathcal{P}$ be absolutely continuous with   density $\eta = \frac{d\mu}{dm}$. Put $f= g\cc$ and $\nu =
 f_\ast m = \mu\cc$.

(i) If $\eta>0$ a.s. then the measure $\nu$ is absolutely continuous with  density $\rho = \frac{d\nu}{dm}>0$ satisfying
\begin{equation*}
 \eta(x) \cdot \rho(f(x)) = \rho(x) \cdot \eta(g(x)) = 1 \qquad \text{ for a.e. } x \in M.
\end{equation*}
(ii) If $\nu$ is absolutely continuous then $f(g(x)) = g(f(x)) =x$ for a.e. $x\in M$.

(iii) Under the previous assumption the Jacobian $\det Df(x)$ and $\det Dg(x)$ exist for almost every $x \in M$ and satisfy
\begin{equation*}
 \det D f(g(x)) \cdot \det D g(x) = \det D f(x) \cdot \det D g(f(x)) = 1,
\end{equation*}
\begin{equation*}
\sigma(x)\cdot\eta(x) = {\sigma(f(x))}\cdot \det Df(x),
\qquad
\sigma(x)\cdot\rho(x) = {\sigma(g(x))}\cdot\det Dg(x)
\end{equation*}
for almost every $x \in M$ where $\sigma = \frac{dm}{d {\mathrm vol}}$ denotes the density of the reference measure $m$ with respect to the Riemannian volume
measure $\vol$.
\end{proposition}

\begin{proof}
 (i) For each Borel function $v:M \rightarrow \R_+$
$$\int_M v \, d\nu = \int_M v\circ f \, dm = \int_M v\circ f \cdot \frac{1}{\eta} \, d\mu = \int_M v\circ f \cdot \frac{1}{\eta(g\circ f)} \, d\mu = \int_M v \cdot
\frac{1}{\eta \circ g} \,dm.$$

Hence, $\nu$ is absolutely continuous with respect to $m$ with density $\frac{1}{\eta \circ g}$.
Interchanging the roles of $\mu$ and $\nu$ (as well as $f$ and $g$) yields the second claim.

(ii), (iii) Part of Brenier- McCann representation result of optimal transports.
\end{proof}

\begin{corollary}
 Under the assumption $\eta>0$ of the previous Proposition:
\begin{equation*}
 \Ent(\mu\cc \mid m) = \Ent(m \mid \mu).
\end{equation*}
\end{corollary}

\begin{proof}
 With notations from above
$$\Ent (\mu\cc \mid m) = \int \rho \log \rho \, dm = \int \frac{1}{\eta \circ g} \log \frac{1}{\eta \circ g} \, dm = \int \frac{1}{\eta} \log \frac{1}{\eta} \, d\mu
=
\Ent (m \mid \mu).$$
\end{proof}

\begin{lemma}
 The conjugation map $$\Conj_\K : \K \rightarrow \K$$ is continuous.
\end{lemma}

\begin{proof}
To simplify notation denote $\Conj_\mathcal{K}$ by $\Conj$.
 Choose a countable dense set $\left\lbrace y_i\right\rbrace_{i \in \mathbb{N}}$ in $M$ and for $k \in \mathbb{N}$ define $\Conj_k :\varphi \mapsto \varphi_k
 \cc$ on $\K$ by $\varphi_k \cc (x) = - \underset{i=1,\ldots,k}{\inf} [\frac{1}{2} d^2(x,y_i)+(\varphi(y_i) ]$.
Then as $k\rightarrow\infty$ $$\varphi_k \cc \nearrow \varphi\cc \qquad \text{ pointwise on M} .$$
Recall that each $\varphi \in \K$ is Lipschitz continuous with Lipschitz constant $D$.

For each $\varepsilon > 0$ choose $k=k(\varepsilon) \in \mathbb{N}$ such that the set $\left\lbrace y_i\right\rbrace _{i=1,\ldots,k(\varepsilon)}$ is an
$\varepsilon$-covering of the compact space $M$. Then
\begin{align*}
&\mid \Conj_k(\varphi)(x) - \Conj(\varphi)(x) \mid  \, \leq
\underset{y \in M}{\sup} \underset{i=1,\ldots,k}{\inf} \mid \frac{1}{2} d^2 (x,y) - \frac{1}{2} d^2(x,y_i) + \varphi(y) - \varphi(y_i) \mid \, \\
&\leq \underset{y \in M}{\sup} \underset{i=1,\ldots,k}{\inf} 2 D\cdot d(y,y_i) \leq 2 D \varepsilon \qquad \textmd{ uniformly in } x \in M \textmd{ and } \varphi \in
\K.
\end{align*}

Now let us consider a sequence $(\varphi_l)_{l \in \mathbb{N}}$ in $\K$ with $\varphi_l \rightarrow \varphi$ in $H^1(M)$. Then for each $k \in \mathbb{N}$ as
$l\rightarrow\infty$
$$\Conj_k (\varphi_l) \rightarrow \Conj_k (\varphi)$$ pointwise on $M$ and thus also in $L^2(M)$.
Together with the previous uniform convergence of $\Conj_k \rightarrow \Conj$ it implies $$ \Conj(\varphi_l) \rightarrow \Conj(\varphi)$$ in $L^2 (M)$ as
$l\rightarrow\infty$. Moreover, we know that $\left\lbrace \Conj (\varphi_l)\right\rbrace _{l\in \mathbb{N}}$ is bounded in $H^1(M)$ (since all gradients are bounded
by $D$).
Therefore, finally $$\Conj(\varphi_l)\rightarrow \Conj (\varphi)$$ in $H^1(M)$ as $ l\rightarrow \infty$. This proves the continuity of $\Conj : \K\rightarrow\K$
with
respect to the $H^1$-norm.
\end{proof}

\begin{theorem}
 The conjugation map $$\Conj_\mathcal{P} :\mathcal{P}\rightarrow \mathcal{P}$$ is continuous (with respect to the weak topology).
\end{theorem}

\begin{proof} Let us first prove continuity of the conjugation map $\Conj_{\tilde \K}: \tilde \K \rightarrow \tilde \K$ (with respect to the $\tilde
H^1$-norm on $\tilde \K$).
Indeed, this follows from the previous continuity result together with the facts that the embedding $H^1\to\tilde H^1,\ \varphi\mapsto\tilde\varphi=\{\varphi+c:
c\in\mathbb{R}\}$ is continuous (trivial fact) and that the map $\tilde H^1\to H^1,\ \tilde\varphi=\{\varphi+c: c\in\mathbb{R}\}\mapsto\varphi-\int_M\varphi dm$ is
continuous (consequence of Poincar\'e inequality).

This in turn implies, due to Proposition 2.5, that the conjugation map $\Conj _\G : \G \rightarrow \G $ is continuous (with respect to the $L^2$-metric on $\G$).
Moreover, due to the same Proposition it therefore also implies that the conjugation map $$\Conj_\mathcal{P} :\mathcal{P}\rightarrow \mathcal{P}$$ is continuous (with
respect to the weak topology).
\end{proof}

\begin{remark}\rm
 In dimension $n=1$, the conjugation map $\Conj_\G :\G \rightarrow \G$ is even an isometry from $\G$, equipped with the $L^1$-metric, into itself.
\end{remark}

\section{Example: The Conjugation Map on $M\subset \mathbb{R}^n$}

In this chapter, we will study in detail the Euclidean case. We assume that $M$ is a compact convex subset of $\mathbb{R}^n$. (The convexity assumption is to simplify notations and results.) The probability measure $m$ is assumed to be absolutely continuous with full support on $M$. 

A function $\varphi:M\to \R$ is $d^2/2$-convex if and only if the function $\varphi_1(x)=\varphi(x)+|x|^2/2$  is {\em convex} in the usual sense:
$$
\varphi_1(\lambda x +(1-\lambda)y)\le\lambda \varphi_1(x)+(1-\lambda)\varphi_1(y)$$
(for all $x,y\in M$ and $\lambda\in[0,1]$) and if its subdifferential lies in $M$:
$$\partial\varphi_1(x)\subset M$$
for all $x\in M$.

A function $\psi$ is the conjugate of $\varphi$ if and only if the function $\psi_1(y)=\psi(y)+|y|^2/2$ is the {\em Legendre-Fenchel transform} of $\varphi_1$:
$$\psi_1(y)=\sup_{x\in M}\left[\langle x,y\rangle-\varphi_1(x)\right].$$
A Borel map $g:M\to M$ is {\em monotone} if and only if
$$\langle g(x)-g(y),x-y\rangle\ge0$$
for a.e. $x,y\in M$. Equivalently, $g$ is monotone if and only if $g=\nabla\varphi_1$ for some convex $\varphi_1:M\to\R$.

\begin{lemma}
(i)
If $\mu=\lambda\delta_z+(1-\lambda)\nu$ then there exists an open convex set $U\subset M$ with $m(U)=\lambda$ such that the optimal transport map $g$ with $g_*m=\mu$ satisfies
$g\equiv z$ a.e. on $U$.

(ii)
The conjugate measure $\mu\cc$ does not charge $U$:
$$\mu\cc(U)=0.$$
\end{lemma}

\begin{proof} (i) Linearity of the problem allows to assume that $z=0$. Let $g=\nabla\varphi_1$ denote the optimal transport map with $\varphi_1$ being an appropriate convex function. Let $V$ be the subset of points in $M$ in which $\varphi_1$ is weakly differentiable with vanishing gradient. By the push forward property it follows that $m(V)=\lambda$. Firstly, then convexity of $\varphi_1$ implies that $\varphi_1$ has to be constant on $V$, say $\varphi_1\equiv\alpha$ on $V$. Secondly, the latter implies that $\varphi_1\equiv\alpha$ on the convex hull $W$ of $V$.
The interior $U$ of this convex set $W$ has volume $m(U)=m(W)\ge m(V)=\lambda$ and $\varphi_1$ is constant on $U$, hence, differentiable with vanishing gradient.
Thus finally $U\subset V$ and $m(U)=\lambda$.

(ii)
Let $\mu_\epsilon$, $\epsilon\in[0,1]$, denote the intermediate points on the geodesic from $\mu_0=\mu$ to $\mu_1=m$. Then $\mu_\epsilon=(g_\epsilon)_*m$ with $g_\epsilon=\exp((1-\epsilon)\nabla\varphi)=\epsilon\cdot Id + (1-\epsilon)\cdot g$ and each $\mu_\epsilon$ is absolutely continuous w.r. to $m$. Hence, $g_\epsilon\cc=g_\epsilon^{-1}$ a.e. on $M$.
Therefore, the conjugate measure $\mu\cc_\epsilon$ satisfies
$$\mu_\epsilon\cc(U)=m\left( (g_\epsilon\cc)^{-1}(U)\right)= m\left(g_\epsilon(U)\right)=\epsilon^n\cdot m(U)=\epsilon^n\cdot\lambda.$$
Now obviously $\mu_\epsilon\to\mu$ as $\epsilon\to0$. According to Theorem 3.6 this implies $\mu\cc_\epsilon\to\mu\cc$ and thus (since $U$ is open)
$$\mu\cc(U)\le\liminf_{\epsilon\to0}\mu_\epsilon\cc(U)=0.$$
\end{proof}

\begin{theorem} (i) If $\mu=\sum_{i=1}^N \lambda_i\delta_{z_i}$ with $N\in\mathbb{N}\cup\{\infty\}$ then there exist disjoint convex open sets $U_i\subset M$ with $m(U_i)=\lambda_i$
such that the optimal transport map $g=\nabla\varphi_1$ with $g_*m=\mu$ satisfies
$g\equiv z_i$ on each of the $U_i$, $i\in\mathbb{N}$.

The measure $\mu\cc$ is supported by the compact $m$-zero set $M\setminus \bigcup_{i=1}^NU_i$.

(ii)
Each of the sets  $U_i$ is the interior of $M\cap A_i$ where
$$A_i=\left\{x\in\R^n: \varphi_1(x)=\langle z_i,x\rangle + \alpha_i\right\}$$
and
$$\varphi_1(x)=\sup_{i=1,\ldots,N}\left[ \langle z_i,x\rangle + \alpha_i\right]$$
with numbers $\alpha_i$ to be chosen in such a way that $m(A_i)=\lambda_i$.

(iii)
If $N<\infty$ then each of the sets $A_i\subset\R^n$, $i=1,\ldots,N$ is a convex polytope. The decomposition $\R^n=\bigcup_{i=1}^N A_i$ is a Laguerre tesselation (see e.g. \cite{LY08} and references therein).

The compact $m$-zero set $M\setminus \bigcup_{i=1}^NU_i$ which supports $\mu\cc$ has finite
$(n-1)$- dimensional Hausdorff measure.
\end{theorem}

\begin{corollary}
 (i) If $\mu$ is discrete then the topological support of $\mu\cc$ is a $m$-zero set. In particular, $\mu\cc$ has no absolutely continuous part.

(ii) If $\mu$ has full topological support then $\mu \cc$ has no atoms.
\end{corollary}

\begin{proof}
(i) Obvious from the previous theorem.

(ii) If $\mu\cc$  had an atom (of mass $\lambda>0$) then according to the previous lemma there would be a convex open set $U$ (of volume $m(U)=\lambda$) such that $\mu(U)=(\mu\cc)\cc(U)=0$.
\end{proof}

\section{The Entropic Measure -- Heuristics}

Our goal is to construct a canonical probability measure $\mathbb{P}^\beta$ on the Wasserstein space
$\mathcal{P} =\mathcal{P}(M)$ over a compact Riemannian manifold, according to the formal ansatz
$$\mathbb{P}^\beta(d\mu) =\frac{1}{Z} e^{-\beta \, \Ent(\mu \mid m)} \mathbb{P}^0(d\mu).$$
Here $\Ent(\cdot \mid m)$ is the \emph{relative entropy} with respect to the reference measure $m$, $\beta$ is a constant $>0$ (`the inverse temperature') and
$\mathbb{P}^0$ should denote a (non-existing) `uniform distribution' on $\mathcal{P}(M)$. $Z$ should denote a normalizing constant. Using the conjugation map
$\Conj_\mathcal{P}: \mathcal{P}(M)\rightarrow \mathcal{P}(M)$ and denoting $\mathbb{Q}^\beta := (\Conj_\mathcal{P})_\ast \mathbb{P}^\beta$, $\mathbb{Q}^0 :=
(\Conj_\mathcal{P})_\ast \mathbb{P}^0$ the above problem can be reformulated as follows:

Construct a probability measure $\mathbb{Q}^\beta$ on $\mathcal{P}(M)$ such that -- at least formally --
\begin{equation}{\label{Qbeta}}
 \mathbb{Q}^\beta (d\nu) = \frac{1}{Z} e^{-\beta \, \Ent (m \mid \nu)} \mathbb{Q}^0(d\nu)
\end{equation}
with some `uniform distribution' $\mathbb{Q}^0$ in $\mathcal{P}(M)$.
Here, we have used the fact that $$\Ent (\nu\cc \mid m) = \Ent (m \mid \nu)$$ (Corollary 3.4), at least if $\nu \ll m$ with $\frac{d \nu}{dm}>0$ almost everywhere.

Probability measures ${\bf{P}} (d\mu)$ on $\mathcal{P}(M)$ -- so called \emph{random probability measures} on $M$ -- are uniquely determined by the distributions
${\bf{P}}_{M_1,\ldots,M_N}$ of the random vectors $$(\mu(M_1),\ldots,\mu(M_N))$$ for all $ N\in \mathbb{N}$ and all measurable partitions $M =
{\dot\bigcup}_{i=1}^N M_i$ of $M$ into disjoint measurable subsets $M_i$.
\medskip
Conversely, if a consistent familiy ${\bf{P}}_{M_1,\ldots,M_N}$ of probability measures on $[0,1]^N$ (for all $ N\in \mathbb{N}$ and all measurable partitions $M =
{\dot\bigcup}_{i=1}^N M_i $) is given then there exists a random probability measure ${\bf P}$ such that $$ {\bf P}_{M_1,\ldots,M_N} (A)= {\bf{P}}((\mu(M_1),\ldots,\mu(M_N)) \in
A)$$ for all measurable $A \subset [0,1]^N$, all $ N\in \mathbb{N}$ and all partitions $M = \overset{.}{\bigcup}_{i=1}^N M_i$.

Given a measurable partition $M = \overset{.}{\bigcup}_{i=1}^N M_i$ the ansatz (\ref{Qbeta}) yields the following characterization of the finite dimensional
distribution on $[0,1]^N$
\begin{equation}{\label{QbetaM}}
 \mathbb{Q}^\beta_{M_1,\ldots,M_N} (d x) = \frac{1}{Z_N} e^{-\beta S_{M_1,\ldots,M_N}(x)} q_{M_1,\ldots,M_N} (dx)
\end{equation}
where $S_{M_1,\ldots,M_N}(x)$ denotes the conditional expectation (with respect to $\mathbb{Q}^0$) of $S(\cdot) = \Ent (m \mid \cdot \,)$ under the condition
$\nu(M_1)=x_1, \ldots, \nu(M_N)=x_N$.

Moreover,  $q_{M_1,\ldots,M_N}(dx) = \mathbb{Q}^0 ((\nu(M_1), \ldots , \nu(M_N)) \in dx)$ denotes the distribution of the random vector
$(\nu(M_1), \ldots, \nu(M_N))$ in the simplex \\
$\sum_N = \left\lbrace x \in [0,1]^N : \sum_{i=1}^N x_i =1 \right\rbrace .$
According to our choice of $\mathbb{Q}^0$, the measure $q_{M_1,\ldots,M_N}$ should be the `uniform distribution' in the simplex $\sum_N$. In \cite{SR08} we argued
that the canonical choice for a `uniform distribution' in $\sum_N$ is the measure
\begin{equation}{\label{qN}}
 q_N(dx)= c\cdot \frac{dx_1 \ldots dx_{N-1}}{x_1 \cdot x_2 \cdot \ldots \cdot x_{N-1} \cdot x_N} \cdot \delta_{(1-{\overset{N-1}{\underset{i=1}{\sum}} x_i})} (d
 x_N).
\end{equation}

It remains to get hands on $S_{M_1,\ldots,M_N}(x)$, the \emph{conditional expectation} of $S(\cdot) =\Ent(m \mid \cdot \,)$ under the constraint $\nu(M_1)=x_1,
\ldots, \nu(M_N)=x_N$. We simply replace it by $\underline{S}_{M_1,\ldots,M_N}(x)$, the \emph{minimum} of $\nu\mapsto\Ent(m \mid \nu)$ under the constraint
$\nu(M_1)=x_1, \ldots, \nu(M_N)=x_N$.

Obviously, this minimum is attained at a measure with constant density on each of the sets $M_i$ of the partition, that is $$ \nu = \sum_{i=1}^{N} \frac{x_i}{m(M_i)}
1_{M_i} m.$$
Hence,
\begin{equation}
\underline{S}_{M_1,\ldots,M_N}(x) = -\sum_{i=1}^{N} \log \frac{x_i}{m(M_i)} \cdot m(M_i).
\end{equation}
Replacing $\underline{S}_{M_1,\ldots,M_N}$ by $S_{M_1,\ldots,M_N}$ in (\ref{QbetaM}), the latter yields
\begin{align*}
  &\mathbb{Q}^\beta_{M_1,\ldots,M_N} (dx) = c \cdot e^{-\beta \underline{S}_{M_1,\ldots,M_N}(x)} q_N(dx) \\
&= \frac{\Gamma(\beta)}{\overset{N}{\underset{i=1}{\prod}} \Gamma (\beta m (M_i))} \cdot x_1^{\beta \cdot m (M_1)-1} \cdot \ldots \cdot x_{N-1}^{\beta \cdot m
(M_{N-1})-1} \cdot x_N^{\beta \cdot m (M_N)-1} \times \\
&\quad \times \delta_{(1-{\overset{N-1}{\underset{i=1}{\sum}}} x_i)}(dx_N) dx_{N-1} \ldots dx_1.
\end{align*}

This, indeed, defines a projective family! Hence, the random probability measure $\mathbb{Q}^\beta$ exists and is uniquely defined. It is the well-known
\emph{Dirichlet-Ferguson process}. Therefore, in turn, also the random probability measure $\mathbb{P}^\beta = (\Conj_\mathcal{P})_\ast \mathbb{Q}^\beta$ exists
uniquely.

\section{The Entropic Measure -- Rigorous Definition}

\begin{definition}
 Given any compact Riemannian space $(M,d,m)$ and any parameter $\beta > 0$ the entropic measure
$$ \mathbb{P}^\beta := (\Conj_\mathcal{P})_\ast \mathbb{Q}^\beta$$
is the push forward of the Dirichlet-Ferguson process $\mathbb{Q}^\beta$ (with reference measure $\beta m$) under the conjugation map
$\Conj_\mathcal{P} :\mathcal{P}(M) \rightarrow \mathcal{P}(M)$.
\end{definition}

$\mathbb{P}^\beta$ as well as $\mathbb{Q}^\beta$ are probability measures on the compact space $\mathcal{P}= \mathcal{P}(M)$ of probability measures on $M$.
Recall the definition of the Dirichlet-Ferguson process $\mathbb{Q}^\beta$ \cite{Ferguson73}: For each measurable
partition $M = \overset{.}{\bigcup}_{i=1}^N M_i$
the random vector $(\nu(M_1),\ldots,\nu(M_N))$ is distributed according to a Dirichlet distribution with parameters $(\beta \,m(M_1),\ldots,\beta\, m(M_N))$.
That is, for any bounded Borel function $u: \R^N \rightarrow \R$
\begin{align*}&\int_{\mathcal{P}(M)} u (\nu(M_1),\ldots,\nu(M_N))\mathbb{Q}^\beta (d\nu) =\\
&\dfrac{\Gamma(\beta)}{{\overset{N}{\underset{i=1}{\prod}}} \Gamma (\beta m(M_i))} \cdot \int_{\left[ 0,1\right]^N  } u(x_1,\ldots,x_N) \cdot x_1^{\beta \, m
(M_1)-1} \cdot \ldots \cdot x_N^{\beta \, m (M_N)-1} \times \\
& \times\delta_{(1-{\overset{N-1}{\underset{i=1}{\sum}}}x_i)}(d x_N) d x_{N-1}\ldots d x_1.
\end{align*}
The latter uniquely characterizes the `random probability measure' $\mathbb{Q}^\beta$.
The existence (as a projective limit) is guaranteed by Kolmogorov's theorem.
$\medskip$

An alternative, more direct construction is as follows:
Let $(x_i)_{i\in\mathbb{N}}$ be an iid sequence of points in $M$, distributed according to $m$, and let $(t_i)_{i\in\mathbb{N}}$ be an iid sequence of numbers in
$[0,1]$, independent of the previous sequence and distributed according to the Beta distribution with parameters 1 and $\beta$, i.e. ${\textmd{Prob}}(t_i \in ds)=
\beta (1-s)^{\beta-1} \cdot 1_{[0,1]}(s) ds$.
Put $$\lambda_k = t_k \cdot \prod_{i=1}^{k-1} (1- t_i) \qquad \text{and} \qquad \nu = \sum_{k=1}^{\infty} \lambda_k \cdot \delta_{x_k}.$$
Then $\nu \in \mathcal{P}(M)$ is distributed according to $\mathbb{Q}^\beta$ \cite{Seth94}.

The distribution of $\nu$ does not change if one replaces the above `stick-breaking process' $(\lambda_k)_{k\in\mathbb{N}}$ by the `Dirichlet-Poisson process' $(\lambda_{(k)})_{k\in\mathbb{N}}$ obtained from it by ordering the entries of the previous one according to their size: $\lambda_{(1)}\ge\lambda_{(2)}\ge\ldots\ge0$.
Alternatively, the Dirichlet-Poisson process can be regarded as the sequence of jumps of a Gamma process with parameter $\beta$, ordered according to size.

Note that $m(M_0)=0$ for a given $M_0\subset M$ implies that $\nu(M_0)=0$ for $\mathbb{Q}^\beta$-a.e. $\nu \in \mathcal{P}(M)$. On the other hand, obviously,
$\mathbb{Q}^\beta$-a.e. $\nu \in \mathcal{P}(M)$ is discrete. In contrast to that, as a corollary to Theorem 4.3 and in analogy to the 1-dimensional case we
obtain:

\begin{corollary} If $M\subset \mathbb{R}^n$ then
 $\mathbb{P}^\beta$-a.e. $\mu \in \mathcal{P}(M)$ has no absolutely continuous part and no atoms.
 The topological support of $\mu\cc$ is a $m$-zero set.
\end{corollary}

For  $\mathbb{P}^\beta$-a.e. $\mu \in \mathcal{P}(M)$ there exist a countable number of open convex sets $U_k\subset M$ (`holes in the support of $\mu$')
with sizes $\lambda_k=m(U_k)$, $k\in\mathbb{N}$. The measure $\mu$ is supported on  the complement of all these holes $M\setminus\bigcup_k U_k$, a compact $m$-zero set.

The sequence $(\lambda_k)_{k\in\mathbb{N}}$ of sizes of the holes is distributed according to the stick breaking process with parameter $\beta$. In particular,
$$\mathbb{E}\lambda_k=\frac1\beta\left(\frac\beta{1+\beta}\right)^k\qquad(\forall k\in\mathbb{N}).$$
In average,  each hole  has size $\le\frac1{1+\beta}$. For large $\beta$, the size of the $k$-th hole decays like $\frac1\beta\exp(-k/\beta)$ as $k\to\infty$.
For small $\beta$, $\lambda_{(1)}$ the size of the largest hole is of order $\sim\frac1{1+0.7\beta}$, \cite{Griffiths88}.

\begin{remark}\rm
 In principle, the reference measures in the conjugation map (see Remark \ref{two-m}) and in the Dirichlet-Ferguson process could be chosen different from each
 other.
\end{remark}

Given a diffeomorphism $h:M\rightarrow M$ the challenge for the sequel will be to deduce a \emph{change of variable formula} for the entropic measure
$\mathbb{P}^\beta (d \mu)$ under the induced transformation  $$\mu \mapsto h_\ast \mu$$ of $\mathcal{P}(M)$.

\begin{conjecture}
 For each $\varphi^2$-diffeomorphism $h : M\rightarrow M$ there exists a function $Y_h ^\beta : \mathcal{P}\rightarrow \R$ such that
\begin{equation} \int U(h_*\mu) \mathbb{P}^\beta (d\mu) = \int U(\mu) Y_h^\beta (\mu) \mathbb{P}^\beta (d\mu),\end{equation}
for all bounded Borel functions $U:\mathcal{P}\rightarrow\R$. (It suffices to consider $U$ of the form $U(\mu)= u(\mu(M_1)),\ldots,\mu(M_N))$ for measurable
partitions $M=\bigcup M_i$ and bounded measurable $u:\R^N\rightarrow \R$.)
The density $Y_h^\beta$ is of the form
\begin{equation}Y_h^\beta(\mu) = \exp\left(\beta \int_M \log \det Dh (x) \mu(dx)\right) \cdot Y_h^0 (\mu)\end{equation}
with $Y_h^0 (\mu)$ being independent of $\beta$.
\end{conjecture}

As an intermediate step, in order to derive a more direct representation for the entropic measure
$\mathbb{P}^\beta$ on $\mathcal{P}(M)$, we may consider the measure $$\mathbb{Q}_\G ^\beta := (\chi^{-1})_\ast \mathbb{P}^\beta
=(\Conj_\G\circ\chi^{-1})_\ast \mathbb{Q}^\beta$$ on $\G$. It is the unique probability measure on the space $\G$ of monotone maps with the property that
\begin{align*}
 &\int_\G u(m((g\cc)^{-1}(M_1)),\ldots,m((g\cc)^{-1}(M_N)) \mathbb{Q}_\G ^\beta (dg) = \\ &\dfrac{\Gamma(\beta)}{{\overset{N}{\underset{i=1}{\prod}}} \Gamma (\beta
 m(M_i))} \cdot \int_{\left[ 0,1\right]^N  } u(x_1,\ldots,x_N) \cdot x_1^{\beta \, m (M_1)-1} \cdot \ldots \cdot x_N^{\beta \, m (M_N)-1} \times \\
& \times \delta_{(1-{\overset{N-1}{\underset{i=1}{\sum}}}x_i)}(d x_N) d x_{N-1}\ldots d x_1
\end{align*}
for each measurable partition $M= \overset{.}{\bigcup}_{i=1}^{N} M_i$ and each bounded Borel function $u:\R^N \rightarrow \R$.
Actually, one may assume without restriction that the partition consists of continuity sets of $m$ (i.e. $m(\partial M_i)=0$ for all $i=1,\ldots,N$) and that $u$ is
continuous.
Note that $(g\cc)^{-1} =g$ almost everywhere whenever $g_\ast m\ll m$.
Moreover, note that in dimension 1, say $M=[0,1]$, the map $\Conj_\G\circ\chi^{-1}: \mathcal{P}\to\G$ assigns to each probability measure $\nu$ its cumulative
distribution function $g$.

\bigskip

In dimension 1, the change of variable formula (6.1) allows to prove closability of the Dirichlet form
 \begin{eqnarray*}
\mathcal{E}(u,u)&=& \int_{\mathcal{P}} \|\nabla u\|^2(\mu) \ d\mathbb{P}^\beta(\mu)\end{eqnarray*}
and to construct the {\em Wasserstein diffusion} $(\mu_t)_{t\ge0}$, the reversible Markov process with continuous trajectories (and invariant distribution
$\mathbb{P}^\beta$) associated to it \cite{SR08}.
The change of variable formula in dimension 1 can also be regarded as a `Girsanov type theorem' for the (normalized) Gamma process \cite{RYZ07}.
Until now, no higher dimensional analogue is known.

The Wasserstein diffusion on 1-dimensional spaces satisfies a logarithmic Sobolev inequality \cite{Stannat}; it can be obtained as scaling limit of empirical
distributions of interacting particle systems \cite{AR07}.

\end{document}